\documentclass[12pt]{article}
\usepackage{latexsym}
\usepackage{amsfonts}
\usepackage{amsmath}
\usepackage{amsmath,amssymb,latexsym}

\makeindex

\newcommand{\imp}{\!\rightarrow\!}
\newcommand{\imps}{\!\rightarrow\!}

\newcommand{\pa}{{\sf PA}}
\newcommand{\gl}{{\sf GL}}
\newcommand{\gla}{{\sf GLA}}

\newcommand{\glacs}{{\sf GLA}_{\mbox{\it\tiny CS}}}

\newcommand{\glae}{{\sf GLA}_{\tiny\emptyset}}

\newcommand{\lp}{{\sf LP}}

\newcommand{\ai}[1]{ {#1}^{\ast} }
\newcommand{\Ai}[1]{ ({#1})^{\ast} }

\newcommand{\proves}{\vdash}

\newcommand{\lc}[1]{#1\!\!:\!\!}

\newcommand{\co}[1]{! {#1} }

\newcommand{\Provable}[1]{ \hbox{\it Provable\/}{#1} }
\newcommand{\Proof}[1]{ \hbox{\it Proof\/}({#1}) }

\newcommand{\PA}{{\sf PA}}

\newtheorem{Prop}{\bf Proposition}

\newtheorem{Theor}{\bf Theorem}
\newenvironment{theorem}{\begin{Theor}\em }{\end{Theor}}
\newtheorem{Lemma}{\bf Lemma}
\newenvironment{lemma}{\begin{Lemma}\em }{\end{Lemma}}
\newtheorem{Coro}{\bf Corollary}
\newenvironment{corollary}{\begin{Coro}\em }{\end{Coro}}
\newtheorem{Fact}{\bf Fact.}

\newtheorem{Remark}{\bf Remark}

\newtheorem{Claim}[enumi]{Claim}

\newtheorem{defin}{\bf Definition}
\newenvironment{definition}{\begin{defin} \em}{\end{defin}}
\newtheorem{exam}{\bf Example}

\newtheorem{notat}{\bf Notation.}

\newenvironment{proof}{{\bf Proof.}}{\hfill $\slot$}
\newcommand{\slot}{\hfill \mbox{$\Box$}\vspace{\parskip}\\}
\newtheorem{Comment}{\bf Comment}

\begin{document}

\title{On a  Hierarchy of Reflection Principles in Peano Arithmetic}

\author{Elena Nogina\thanks{Supported by PSC CUNY Research Awards program.} \\ \\
 {\small BMCC CUNY, Department of Mathematics}\\
{\small 199 Chambers Street, New York, NY 10007}\\
{\small {\tt E.Nogina@gmail.com}} }

\date{\empty}
\maketitle
\begin{abstract}  
We study reflection principles of Peano Arithmetic {\sf PA} which are based on both proof and provability. Any such reflection principle in {\sf PA} is equivalent to either $\Box P\!\rightarrow\! P$ ($\Box P$ stands for {\it $P$ is provable}) or $\Box^k u\!\!:\!\!P\!\rightarrow\! P$ for some $k\geq 0$ ($\lc{t}P$ states {\it $t$ is a proof of $P$}). Reflection principles constitute a non-collapsing hierarchy with respect to their deductive strength 
$$u\!\!:\!\!P\!\rightarrow\! P\ \ \prec\  \ \Box u\!\!:\!\!P\!\rightarrow\! P\ \ \prec\  \ \Box^2 u\!\!:\!\!P\!\rightarrow\! P \ \ \prec\ \ldots\ \prec\  \ \Box P\!\rightarrow\! P.$$
 
\end{abstract}

\medskip\par
\section{Introduction}

Reflection Principles are classical objects in Proof Theory. They were introduced by Rosser \cite{Ros36} and Turing \cite{Tur39} in the 1930s, and later studied by Feferman \cite{Fef60,Fef62}, Kreisel and L\'{e}vi \cite{KL68}, Schmerl \cite{Sch79}, Artemov \cite{Art87}, Beklemishev \cite{Bek97,Bek03}, and many others (cf. survey \cite{AB05}).
 
A  \emph{proof predicate} is a provably decidable formula
\emph{Proof} that enumerates all theorems of {\sf PA},
$${\sf PA}\proves\varphi \ \ \ \ \mbox{ iff }\ \ \ \ \Proof{k,\varphi}  \ \mbox{for some $k$.}$$ 
In this paper all proof predicates are assumed \emph{normal} (\cite{Art01a}), t.e.  \\ 
1. for every $k$ set $$T(k)=\{\varphi\mid\Proof{k, \varphi}\}$$ is finite, the function from
$k$ to $T(k)$ is computable; \\ \\
2. for any $k$ and $l$ there is $n$ such that
\[ T(k)\cup T(l)\subseteq T(n) .\]
Prime example: G\"odel's proof predicate. 

A natural example of a Reflection Principle is given by so-called local (or implicit) reflection. Let $\Provable{\ F}$ be $\exists x\Proof{x,F}$.
In the formal provability setting, the local reflection principle is the set of all arithmetical formulas 
$$
\mbox{\it Provable $F$ $\imps\ F$,}
$$
where $F$ is an arithmetical formula.
Though all the instances of this reflection principle are true in the standard model of Peano Arithmetic \pa, some of them are not provable. For example, if $F$ is falsum $\bot$, the local reflection principle becomes G\"odel's consistency formula
$$
\mbox{\it $\neg$Provable $\bot$.}
$$

Another example is given by the explicit reflection principle, i.e., the set of formulas 
\[
\Proof{t,F}\imps F
\]
where $t$ is an arbitrary proof term, and $F$ an arithmetical formula. Here the situation is quite different; all instances of explicit reflection are provable.

Indeed, if $\Proof{t,F}$ holds, then $F$ is obviously provable in ${\sf PA}$, and
so is formula $\Proof{t,F}\imp F$. If $\neg\Proof{t,F}$ holds, then it is provable in ${\sf PA}$ (since $\neg\Proof{x,y}$ is decidable) and
$\Proof{t,F}\imps F$ is again provable.

We study (cf. \cite{Nog14b}) reflection principles of Peano Arithmetic {\sf PA} which are based on both proof and provability predicates. (cf. \cite{Art01a,Boo93}). 

Let $P$ be a propositional letter and each of $Q_1,Q_2,\ldots ,Q_m$ is either `$\Box$' standing for provability in {\sf PA}, or `$u\!:$' standing for 
\[\mbox{`{\it $u$ is a proof of $\ldots$ in {\sf PA}}'},\] $u$ is a fresh proof variable. Then the formula
$$Q_1Q_2 \ldots Q_m P\!\rightarrow\! P$$ 
is called {\it generator}, and the set of all its arithmetical instances is the {\it reflection principle} corresponding to this generator. We will refer to reflection principles using their generators. 

It is immediate that all reflection principles without explicit proofs ($Q_i=\Box$ for all $i$) are equivalent to the local reflection principle $\Box P\!\rightarrow\! P$. All $\Box$-free reflection principles are provable in {\sf PA} and hence equivalent to $u\!\!:\!\!P\!\rightarrow\! P$. Mixing explicit proofs and provability yields infinitely many new reflection principles: 

1.  {\it Any reflection principle in {\sf PA} is equivalent to either $\Box P\!\rightarrow\! P$ or
 $\Box^k u\!\!:\!\!P\!\rightarrow\! P$ for some $k\geq 0$.}

2.  {\it Reflection principles constitute a non-collapsing hierarchy with respect to their deductive strength
$$u\!\!:\!\!P\!\rightarrow\! P\ \prec\ \Box u\!\!:\!\!P\!\rightarrow\! P\ \prec\ \Box^2 u\!\!:\!\!P\!\rightarrow\! P\ \prec\ \ldots\ \prec\ \Box P\!\rightarrow\! P.$$}
The proofs essentially rely on introduced by the author G\"odel-L\"ob-Art\"emov logic \gla\ of formal provability and explicit proofs. 

\section{Description and basic properties of \gla}

We describe the logic \gla\ introduced in \cite{Nog06} (see also \cite{Nog14a}) in the union of the original languages of 
G\"odel-L\"ob Logic \gl (cf. \cite{Boo93, Sol76}) and Artemov's Logic of Proofs \lp (\cite{Art01a}).  

The following two systems were predecessors of \gla: 
\begin{itemize}
\item
system ${\sf B}$ from \cite{Art94}, which did not have operations on proofs; 
\item
system ${\sf LPP}$  from \cite{Sid97a,Y_S01} in an extension of languages of the logic of formal provability \gl\  and the Logic of Proofs \lp.
\end{itemize}
The immediate successors of \gla\ are the logic {\sf GrzA} (\cite{Nog09}) of strong provability and explicit proofs  and symmetric logic of proofs and provability (\cite{Nog10}).
\medskip\par\noindent
{\bf Language of \gla.}
\medskip\par\noindent
\emph{Proof terms} are built from \emph{proof variables} $x,y,z,\dots$ and \emph{proof constants}
$a,b,c,\dots$ by means of two binary operations: \emph{application}~`$\cdot$' and \emph{union}~`$+$', and one
unary \emph{proof checker}~`$!$'.

Formulas of {\sf GLA} are defined by the grammar
\[ A=S\mid A\imp A \mid A\wedge A \mid A\vee A \mid \neg A \mid \Box A \mid \lc{t}A\ ,\]
where $t$ stands for any proof term and $S$ for any sentence letter.

Axioms and rules of both G\"odel-L\"ob logic {\sf GL} and {\sf LP}, together with three specific principles
connecting explicit proofs with formal provability, constitute $ \sf \glae$.

\medskip\par\noindent
I. {\bf Axioms of classical propositional logic} \medskip\par
Standard axioms of the classical logic (e.g., {A1-A10} from \cite{Kle52})
\medskip\par\noindent
II. {\bf Axioms of Provability Logic {\sf GL}}\medskip\par
{\bf GL1} $\Box(F\imp G)\imp(\Box F\imp \Box G)$ \hfill {\em Deductive Closure/Normality}\par
{\bf GL2} $\Box F \imp \Box\Box F$ \hfill {\em Positive Introspection/Transitivity}\par
{\bf GL3} $\Box (\Box F \imp F)\imp \Box F$ \hfill {\em L\"ob Principle}\par
\medskip\par\noindent
III. {\bf Axioms of the Logic of Proofs $\lp$}
\medskip\par {\bf LP1} $\lc{s}(F\imp G)\ \imp\ (\lc{t}F\imp \lc{[s\!\cdot\!t]}G)$ \hfill
\emph{Application}
\par {\bf LP2} $\lc{t}F\ \imp\ \lc{\co{t}}(\lc{t}F)$ \hfill \emph{Proof Checker}
\par {\bf LP3} $\lc{s}F\imp\lc{[s\!+\!t]}F$, $\ \ \lc{t}F\imp\lc{[s\!+\!t]}F$ \hfill \emph{Sum}
\par {\bf LP4} $\lc{t}F\imp F$ \hfill \emph{Explicit Reflection}
\medskip\par\noindent
IV. {\bf Axioms connecting explicit and formal provability}
\medskip\par {\bf C1} $\lc{t}F\imp \Box F$ \hfill \emph{Explicit-Implicit connection} \par
{\bf C2} $\neg\lc{t}F\imp \Box\neg\lc{t}F$ \hfill \emph{Explicit-Implicit Negative Introspection} \par
{\bf C3} $\lc{t}\Box F\imp F$ \hfill \emph{Explicit-Implicit Reflection} \par
\medskip\par\noindent
V.  {\bf Rules of inference}
\par {\bf R1} $F\imp G,\ F\proves G$ \hfill \emph{Modus Ponens} \par
{\bf R2} $\vdash F\ \Rightarrow\ \vdash \Box F$ \hfill \emph{Necessitation}\par
{\bf R3} $\vdash \Box F \Rightarrow\ \vdash F$ \hfill \emph{Reflection Rule}\par
 \medskip\par\noindent
A {\bf Constant Specification} ${C\!S}$ for $\gla$ is the set of formulas 
\[ \{\lc{c_1}A_1,\lc{c_2}A_2,\lc{c_3}A_3,\ldots\}, \]
where each $A_i$ is an axiom of $\glae$ and each $c_i$ is a proof constant.
$$\glacs\ = \glae\ + C\!S,$$
$$\gla\ = \glacs\ \mbox{with the ``total" \it$C\!S$}. $$

One of the principal properties of \gla\ is its ability to internalize its own proofs \cite{Nog14a}: 
\emph{If $\gla\proves F$, then for some proof term $p$, $\gla\proves\lc{p}F$}. 

An arithmetical interpretation $\ast$ of a \gla-formula is the direct sum of corresponding arithmetical interpretations for $\gl$ and $\lp$; in particular, 
$$\Ai{\Box G}  =  \Provable\ {\ai{G}};$$
$$\Ai{\lc{p}F}\ \ =\ \ \Proof{\ai{p},\ai{F}}.$$

$\gla$ is sound with respect to the arithmetical provability interpretation (\cite{Nog06,Nog14a}): 
\par
{\it For any Constant Specification $CS$ and any arithmetical interpretation $\ast$ respecting $CS$, if $\glacs\proves F$ then $\pa\proves F^\ast$.}

The following arithmetical completeness theorem holds (\cite{Nog06,Nog14a}): 
\par
{\it For any finite constant specification $C\!S$, if 
$\glacs\not\proves F$, then for some interpretation $\ast$ respecting $C\!S$, $\pa\not\proves\ai{F}$.}

In \cite{Nog07,Nog14a}, \gla\ was supplied with Kripke-style semantics and found to be complete with respect to it. 

\section{Reflection principles in Peano Arithmetic}

Fix a normal proof predicate {\it Proof} and, therefore, the corresponding provability predicate {\it Provable}. If $F$ is a \gla-formula, then $\{\ai{F}\}$ denotes the set of all arithmetical interpretations of $F$ based on {\it Proof} and {\it Provable}. 

\begin{definition} Let $P$ be a propositional letter and each of $Q_1,Q_2,\ldots ,Q_m$ be either  $\Box$ or `$\lc{u}\ $' for some fresh proof variable $u$. Then a formula
\[
Q_1Q_2\ldots Q_m P \imps P
\]
is called a {\it generator} and the set $\{[Q_1Q_2\ldots Q_mP \imps P]^\ast\}$ is a {\it reflection principle} corresponding to this generator.

\end{definition}
For example, the implicit reflection principle is generated by \gla-formula $\Box P \imps P$,  the explicit reflection is generated by $\lc{u}P\imps P$. 

\begin{definition} Let $G$ and $H$  be {\sf GLA}-formulas. We say that $\{\ai{H}\}\preceq  \{\ai{G}\}$, or $H\preceq  G,$ for short, if $\PA+\{\ai{G}\}$ proves all formulas from $\{\ai{H}\}$. $H \simeq G$ (is read as ``\emph{$H$ is equivalent to $G$}'') means that both $H\preceq  G$ and $G\preceq  H$ hold; $H\prec G$ stands for ($H\preceq  G$ and $H\not\simeq G$).
\end{definition}

Example: $$\lc{u}P\imps P\ \ \ \prec \ \  \Box P \imps P. $$

We study the structure of reflection principles in the explicit-implicit language. In particular, we establish 
classification of reflection principles (Theorem \ref{tfour}):
\par
\emph{Any reflection principle is equivalent to either $\Box P\imps P$ or, for some $k\geq  0$, to $\Box^k\lc{u}P\imps P$}.
\medskip\par\noindent
We also discover that
reflection principles constitute a hierarchy (Theorem \ref{nine-four}):
\par
$$\lc{u}P\imps P\ \prec\ \Box\lc{u}P\imps P\ \prec\ \Box^2\lc{u}P\imps P\ \prec\ \ldots\ \prec\ \Box P\imps P. $$
These two results could be immediately concluded from the well-known fact  (Lemma \ref{nine-five}):
$$ \neg\bot\ \prec\ \neg\Box\bot\ \prec\ \neg\Box^2\bot\ \prec\ \ldots\ \prec \Box P\imps P $$
together with the following assertions we will establish in this section: 
\begin{enumerate}
\item \emph{For each $n\geq 1$,} $\Box^n P\imps P \ \ \simeq\ \  \Box P\imps P$ ({Uniqueness of Provability Reflection}, Theorem \ref{nine-one});
\item \emph{For $k\geq 0$,}  $\Box^k\lc{u}Q_1Q_2\ldots Q_n P \imps P \ \ \simeq\ \ \Box^k\lc{u}P\imps P$ (Theorem \ref{generalk});
\item  \emph{For $k\geq 0$,} $\ \ \Box^k\lc{u}P\imps P\ \ \simeq \ \ \neg\Box^k\bot$, (Theorem \ref{six}).
\end{enumerate}

\subsection{Uniqueness of Provability Reflection}

Let $Q_1Q_2\ldots Q_m P \imps P$ be a generator, and $Q_1Q_2\ldots Q_m$ consists only of implicit provability operators $\Box$. It is obvious that the corresponding principle is equivalent to $\Box P \imps P$.

\begin{theorem}\label{nine-one}{\bf (Uniqueness of Provability Reflection)}
\end{theorem} 
\begin{proof} In light of the arithmetic soundness of $\glae$, $$\Box P\imps P\  \preceq  \ \Box^n P\imps P$$ follows from the fact that $\glae\proves \Box P\imps\Box^n P$. The converse inequality
$\Box^n P\imps P\  \preceq  \ \Box P\imps P$ is implied by the fact that
\[ [\Box^n P\imps \Box^{n-1}P]\wedge[\Box^{n-1} P\imps \Box^{n-2}P]\wedge\ldots \wedge[\Box P\imps P] \imps [\Box^n P\imps P] \]
is derivable in $\glae$.
\end{proof}

\subsection{Leading-Explicit Reflection Principles are provable}

\begin{theorem}\label{e-ref} \emph{For any $n\geq 0$, $\glae\proves\lc{u}Q_1Q_2\ldots Q_n P \imps P$. }
\end{theorem} 
\begin{proof} Induction on $n$. The base case $n=0$ is trivial. For the induction step consider two cases. 

\emph{Case 1}: $Q_1$ is ``$\lc{v}\ $" for some proof variable $v$. Then, by explicit reflection,
\[ \glae\proves\lc{u}Q_1Q_2\ldots Q_n P\imps \lc{v}Q_2\ldots Q_n P. \]
By the Induction Hypothesis, $$\glae\proves\lc{v}Q_2\ldots Q_n P \imps P.$$ Hence
\[ \glae\proves\lc{u}Q_1Q_2\ldots Q_n P \imps P .\]

\emph{Case 2}: $Q_1$ is $\Box$. Then $\lc{u}Q_1Q_2\ldots Q_n P$ has type $\lc{u}\Box^m F$ for some $m\geq 1$, where $F$ is either $P$ or $\lc{w}Q_{n-m-1}\ldots Q_n P$. Now we show that
$\glae\proves \lc{u}\Box^m F\imps F$.
Indeed, \\
\begin{tabular}{lll}
&&\\
1. & $\neg\Box^m F\imp\neg\lc{u}\Box^m F$, & by E-reflection; \\
2. & $\neg\lc{u}\Box^m F\imp\Box(\neg\lc{u}\Box^m F)$, & axiom {\bf C2};\\
3. & $\Box(\neg\lc{u}\Box^m F)\imp\Box(\lc{u}\Box^m F\imp F)$, & by reasoning in {\sf GL};\\
4. & $\neg\Box^m F\imp\Box(\lc{u}\Box^m F\imp F)$, & from 1,2, and 3;\\
5. & $\Box(\lc{u}\Box^m F\imp F)\imp \Box^m(\lc{u}\Box^m F\imp F)$, & from transitivity;\\
6. & $\neg\Box^m F\imp\Box^m(\lc{u}\Box^m F\imp F)$, & from 4 and 5;\\
7. & $\Box^m F\imp\Box^m(\lc{u}\Box^m F\imp F)$, & by reasoning in {\sf GL};\\
8. & $\Box^m(\lc{u}\Box^m F\imp F)$, & from 6 and 7;\\
9. & $\lc{u}\Box^m F\imp F$, & by Reflection Rule.\\
\end{tabular}
\medskip\par\noindent
If $F$ is $P$ we are done; if $F$ is $\lc{w}Q_{n-m-1}\ldots Q_n P$, then, by the Induction Hypothesis, $\glae\proves F\imps P$ which yields the theorem claim as well.
\end{proof}

\begin{corollary}{\bf (Uniqueness of Leading-Explicit Reflection)}\label{k=0} \emph{ Let $\lc{u}Q_1Q_2\ldots Q_n P \imps P$ be a reflection principle generator. Then}
\[ \lc{u}Q_1Q_2\ldots Q_n P \imps P\ \ \simeq\ \ \lc{u}P\imps P.\]

\end{corollary}
\begin{proof} Follows from Theorem~\ref{e-ref} by the arithmetical soundness of $\glae$. 
\end{proof}

\subsection{Classification of Reflection Principles}

\begin{theorem}\label{generalk} \emph{Let $k\geq 0$ and $\Box^k\lc{u}Q_1Q_2\ldots Q_n P \imps P$ be a reflection principle generator. Then 
\[ \Box^k\lc{u}Q_1Q_2\ldots Q_n P \imps P \ \ \simeq\ \ \Box^k\lc{u}P\imps P .\] 
}
\end{theorem}
\begin{proof} The following argument could not be done in \gla; so, we reason in \pa\ instead.

First, we establish ``$\preceq $", i.e., $$\PA^\prime = \PA + \{[\Box^k\lc{u}P\imps P]^\ast\}\proves\{[\Box^k\lc{u}Q_1Q_2\ldots Q_n P \imps P]^\ast\}.$$
Fix an interpretation $\ast$. By Theorem \ref{k=0}, 
\[
\PA\proves \lc{u^\ast}[Q_1Q_2\ldots Q_n P]^\ast\imps \ai{P}.
\]
We write $\lc{t}F$ for $\Proof{t,F}$ and $\Box F$ for {\it Provable F} in \pa, for brevity.

Let $s$ be its proof in \PA. Then,
\[ \PA\proves \lc{s}(\lc{u^\ast}[Q_1Q_2\ldots Q_n P]^\ast\imps \ai{P}). \]
By proof checking and internalized {\it Modus Ponens} in \pa, we can find an arithmetical proof $t$ such that
\[ \PA\proves \lc{u^\ast}[Q_1Q_2\ldots Q_n P]^\ast\imps \lc{t}\ai{P}, \]
from which we conclude
\[ \PA\proves \Box^k\lc{u^\ast}[Q_1Q_2\ldots Q_n P]^\ast\imps \Box^k\lc{t}\ai{P}, \]
\[ \PA^\prime \proves \Box^k\lc{u^\ast}[Q_1Q_2\ldots Q_n P]^\ast\imps \ai{P}, \]
i.e.,
\[ \PA^\prime \proves [\Box^k\lc{u}Q_1Q_2\ldots Q_n P\imps P]^\ast . \]

Let us now establish ``$\succeq $", i.e., that $$\PA^{\prime\prime} = \PA + \{[\Box^k\lc{u}Q_1Q_2\ldots Q_n P\imps P]^\ast\}\proves\{[\Box^k\lc{u}P\imps P]^\ast\}.$$
\begin{lemma}\label{sharp} \emph{For each interpretation $\ast$ there is an interpretation $\sharp$ which coincides with $\ast$ on $P$ such that}
\[ \PA\proves [\lc{u}P]^\ast \imps [\lc{u}Q_1Q_2\ldots Q_n P ]^\sharp .\]
\end{lemma}
\begin{proof} By induction on $n$. The case $n=0$ is trivial. Let for some interpretation $\flat$ coinciding with $\ast$ on $P$,
\[ \PA\proves [\lc{u}P]^\ast \imps [\lc{u}Q_2\ldots Q_n P]^\flat .\]
By proof-checking,
\[
 \PA\proves [\lc{u}P]^\ast \imps\ \lc{!u^\flat}\lc{u^\flat}[Q_2\ldots Q_n P]^\flat .
\]

\emph{Case 1}. If $Q_1$ is a proof variable $v$, then define $u^\sharp$ as $!u^\flat$, $v^\sharp$ as $u^\flat$, set $\sharp$ to be $\flat$ everywhere else, and get the desired
\[ \PA\proves [\lc{u}P]^\ast \imps [\lc{u}\lc{v}Q_2\ldots Q_n P]^\sharp .\]

\emph{Case 2}. If $Q_1$ is $\Box$, then by reasoning in $\PA$ find a proof $t$ such that
\[ \PA\proves \lc{!u^\flat}\lc{u^\flat}[Q_2\ldots Q_n P]^\flat \imps \lc{t}\Box [Q_2\ldots Q_n P]^\flat ,\]
therefore, $\PA\proves [\lc{u}P]^\ast \imps \lc{t}\Box [Q_2\ldots Q_n P]^\flat$.
Define $u^\sharp=t$ ($u$ is fresh!) and set $\sharp$ equal $\flat$ everywhere else. Then
\[ \PA\proves [\lc{u}P]^\ast \imps [\lc{u}Q_1Q_2\ldots Q_n P]^\sharp ,\]
which completes theorem's proof. 
\end{proof}

Now, by the standard $\pa$-reasoning,
\[ \PA\proves [\Box^k\lc{u}P]^\ast \imps [\Box^k\lc{u}Q_1Q_2\ldots Q_n P]^\sharp ,\]
and since
\[ \PA^{\prime\prime}\proves [\Box^k\lc{u}Q_1Q_2\ldots Q_n P\imps P]^\sharp \]
and $P^\sharp=P^\ast$ we conclude that
\[ \PA^{\prime\prime}\proves [\Box^k\lc{u}P\imps P]^\ast . \]
\end{proof}

From Theorems \ref{nine-one}  and \ref{generalk} immediately follows 

\begin{theorem}\label{tfour}{\bf (Classification of Reflection Principles)}\label{refprinciples} \emph{Any reflection principle is equivalent to either \[ \Box P\imps P \] or, for some $k\geq 0$, to \[ \Box^k\lc{u}P\imps P.\] } 
\end{theorem}
\begin{proof} Consider and arbitrary reflection principle $\pi$ 
\[ Q_1Q_2\ldots Q_n P  \imps P. \]
If all $Q_i$ are $\Box$'s, then, by Theorem~\ref{nine-one}, $\pi = \Box P \imps P$. Otherwise, $\pi$ can be written as 
\[ \Box^k\lc{u}Q_{n-k-1}\ldots Q_n P  \imps P \]
for an appropriate $k\geq 0$. By Theorem~\ref{generalk}, $\pi = \Box^k\lc{u}P\imps P$. 
\end{proof}

\subsection{Hierarchy of Reflection Principles}

\begin{theorem}\label{nine-four} \emph{Reflection principles form a linear ordering}
\[ \lc{u}P\imps P\ \prec\ \Box\lc{u}P\imps P\ \prec\ \Box^2\lc{u}P\imps P\  \prec\ \ldots\ \prec\ \Box P\imps P. \]
\end{theorem}

This Theorem is an immediate corollary of the following two assertions. 

\begin{theorem}\label{six}{\it For each $k\geq 0$,}  $\ \ \Box^k\lc{u}P\imps P\ \ \simeq \ \ \neg\Box^k\bot$. 
\end{theorem}
\begin{proof}
Putting $P=\bot$ we get $\Box^k\lc{u}P\imps P\succeq  \neg\Box^k\bot$. For the converse, argue in $\glae$. Case $k=0$ is trivial. Let $k\geq 1$. Assume $\neg\Box^k\bot$, $\Box^k\lc{u}P$, and $\neg P$ and look for a contradiction. By explicit reflection, from $\neg P$ we derive $\neg\lc{u}P$ and, by explicit-implicit negative introspection, $\Box\neg\lc{u}P$. By transitivity, we get $\Box^k\neg\lc{u}P$. From this and $\Box^k\lc{u}P$, by the usual modal reasoning we conclude $\Box^k(\neg\lc{u}P\wedge\lc{u}P)$; hence $\Box^k\bot$, a contradiction.
\end{proof} 

Now, to get Theorem \ref{nine-four}, it suffices to refer to a well-known fact:

\begin{lemma}\label{nine-five}
\[ \neg\bot \prec \neg\Box\bot\prec\neg\Box^2\bot\prec\neg\Box^3\bot\prec \ldots \prec \Box P\imps P. \]
\end{lemma}
\begin{proof} 

a) For $k\geq 1$, by transitivity, $\gl\proves\Box^{k-1}\bot\imp\Box^k\bot$, hence
$\gl\proves\neg\Box^k\bot\imp\neg\Box^{k-1}\bot$.
By the arithmetical soundness of \gl, 
\[ \neg\Box^{k-1}\bot\preceq \neg\Box^k\bot. \]
Modal formula $\Box^k\bot\imp\Box^{k-1}\bot$ is false at the root of a $k$-node linear model, hence not provable in \gl. By the arithmetical completeness of \gl, $\pa\not\proves\Box^k\bot\imp\Box^{k-1}\bot$,
hence $$\neg\Box^{k}\bot\not\preceq \neg\Box^{k-1}\bot,$$ therefore $$\neg\Box^{k-1}\bot \prec \neg\Box^k\bot.$$

b) For each $k\geq 0$, $\neg\Box^k\bot \preceq \Box P \imp P$. Indeed, cases of $k=0,1$ are trivial. Consider $k\geq 2$. From instances of $\Box P\imp P$
\[ \Box^k\bot\imp\Box^{k-1}\bot,\ \Box^{k-1}\bot\imp\Box^{k-2}\bot,\ldots,\Box\bot\imp\bot , \]
by a chain of syllogisms, we derive $\Box^k\bot\imp\bot$, and hence 
$$\pa + \{[\Box P\imp P]^\ast\}\proves\neg\Box^k\bot. $$
 
c) For any $k\geq 0$, $\Box P\imp P\not\preceq \neg\Box^k\bot$. Suppose the opposite, namely, that for some $k\geq 0$, $\Box P\imp P\preceq \neg\Box^k\bot$. Since, by b), \[ \neg\Box^{k+1}\bot\preceq  \Box P\imp P,\] we have $\neg\Box^{k+1}\bot\preceq \neg\Box^k\bot$, which is impossible, by a). 
\end{proof}

\section{Acknowledgements}
The author is grateful to Sergei Artemov, Melvyn Fitting,  Hidenori Kurokawa, Anil Nerode, Junhua Yu and logic groups of National Chung Cheng University, Academia Sinica and the Computational Logic seminar of the CUNY Graduate Center for useful discussions.

\end{document}